\documentclass[smallcondensed]{svjour3}
\usepackage[utf8]{inputenc}
\usepackage[T1]{fontenc}
\usepackage[english]{babel}
\usepackage{amsmath,amssymb,graphicx,subfigure,scrtime,xypic,tikz}
\usetikzlibrary{arrows}
\usepackage[backgroundcolor=white]{todonotes}

\numberwithin{equation}{section}

\newcommand{\struc}[1]{\langle #1 \rangle}
\newcommand{\class}[1]{\mathcal{#1}}
\newcommand{\var}[1]{\class{#1}}

\newcommand{\spec}[1]{{\mathrm{Spec}(#1)}}

\newcommand{\topo}[1]{\mathbf{#1}}

\newcommand{\alg}[1]{\mathbf{#1}}
\newcommand{\Upb}[1]{\mathrm{\mathbf{Up}}(#1)}
\newcommand{\Clo}[1]{\mathrm{\mathbf{Cl}_0}(#1)}
\newcommand{\me}{\mathrm{\bf m}}

\def\tkzscl{0.8} %Tikz scale factor

\title{Conservative median algebras and semilattices}
\author{Miguel Couceiro \and Jean-Luc Marichal \and Bruno Teheux }
\institute{Miguel Couceiro \at LORIA (CNRS - Inria Nancy Grand Est - Université de Lorraine), \' Equipe Orpailleur,
Batiment B, Campus Scientifique,   B.P. 239,  
 F-54506 Vandoeuvre les Nancy. \email{miguel.couceiro@inria.fr}
\and 
Jean-Luc Marichal \at
Mathematics Research Unit, FSTC, University of Luxembourg 
6, rue Coudenhove-Kalergi, L-1359 Luxembourg, Luxembourg. \email{jean-luc.marichal@uni.lu }
\and
Bruno Teheux \at
Mathematics Research Unit, FSTC, University of Luxembourg 
6, rue Coudenhove-Kalergi, L-1359 Luxembourg, Luxembourg. \email{bruno.teheux@uni.lu }}
\begin{document}
%\linenumbers

\maketitle

\begin{abstract}
We characterize conservative median algebras and semilattices by means  of forbidden substructures and by providing their representation as chains. Moreover, using a dual equivalence between median algebras and certain   topological structures, we obtain descriptions of the median-preserving mappings  between products of finitely many chains.    
\end{abstract}

\section{Introduction and preliminaries}

In this paper we are interested in certain algebraic structures called median algebras. 
A \emph{median algebra} is a ternary algebra $\alg{A}=\struc{A, \me}$  that satisfies the following equations
\begin{gather*}
\me(x,x,y)=x,\\
\me(x,y,z)=\me(y,x,z)=\me(y,z,x),\\
\me(\me(x,y,z), t, u)=\me(x,\me(y,t,u), \me(z,t,u)).
\end{gather*}

Median algebras have been investigated by several authors (see  \cite{Birkhoff1947,Grau1947} {for early} references  on median algebras and see \cite{Bandelt1983,Isbell1980} for some surveys). 
For instance, it is shown in \cite{Sholander1954} that  for each element $a$ of a median algebra $\alg{A}$, the relation $\leq_a$ defined on $A$ by
\[
x \leq_a y \quad \iff \quad \me(a,x,y)=x
\]
is a $\wedge$-semilattice order with bottom element $a$. The associated operation  $\wedge$ is  defined by $x\wedge y =\me(a,x,y)$. Semilattices constructed in this way are called \emph{median semilattices}, and can be characterized as follows.

\begin{theorem}[(3.1) in  \cite{Sholander1954}]\label{thm:semi}
A $\wedge$-semilattice is a median semilattice if and only if each of its principal ideal is a distributive lattice, and any three elements have a join whenever each pair of them is bounded above. 
\end{theorem}
In particular, any distributive lattice is a median semilattice. According to Theorem \ref{thm:semi}, we can define a ternary operation $\me_\leq$ called the \emph{median operation of $\leq$} on every median semilattice $\struc{A, \leq}$ by setting
\begin{equation}\label{eqn:intro}
\me_\leq(x,y,z)=(x\wedge y)\vee (x\wedge z)\vee (z\wedge y),
\end{equation}
for every $x,y,z \in A$.
It can be proved \cite[Lemma 3 (6)]{Avann1961} that $\me=\me_{\leq_a}$ for every median algebra $\alg{A}=\struc{A, \me}$, and every $a\in A$.

%%%% I think that the following is not used in the paper.
%Result (10.1) in \cite{Sholander1952}  (see also \cite[p. 137, Theorem 4]{Birkhoff1948}) states that if the median algebra $\alg{A}$ contains two elements $0$ and $1$ such that $\me(0,x,1)=x$ for every $x \in A$, then $(A, \leq_0)$ is a distributive lattice order bounded by $0$ and $1$, and where  $x\wedge y$ and   $x\vee y$ are given by  $\me(x,y,0)$ and  $\me(x,y,1)$, respectively. 
%%%% The following is redudent since a distibutive lattice is a median semilattice 
%Conversely, if $\alg{L}=\struc{L, \vee, \wedge}$ is a distributive lattice, then the term function defined by \eqref{eqn:intro} is denoted by $\me_{\alg{L}}$ and give rises to a median algebra on $L$, called the \emph{median algebra associated with $\alg{L}$}. 
% It is noteworthy that  equations satisfied by median algebras of the form $\struc{L, \me_\leq}$ where $\struc{L, \leq}$ is a bounded distributive lattice are exactly the same as those satisfied by median algebras. For further background see, e.g.,  \cite{Bandelt1983}.

Here, we are particularly interested in median algebras $\alg{A}$ that are \emph{conservative}, i.e.,  that satisfy 
\begin{equation}\label{eqn:cons}
\me(x,y,z)\in \{x,y,z\}, \qquad x, y, z \in A.
\end{equation}
Although condition \eqref{eqn:cons} appears in \S 11 of  \cite{Sholander1952},  to the best of the authors' knowledge, the present work constitutes the first attempt of a systematic study of conservative median algebras.
A median semilattice $\struc{A, \leq}$ whose median operation $\me_\leq$ satisfies \eqref{eqn:cons}  is called a \emph{conservative median semilattice}.
Note that  a median algebra is conservative if and only if each of its subsets  is a median subalgebra. Moreover, if $\alg{L}$ is a chain, then $\me_{\alg{L}}$ satisfies (\ref{eqn:cons}); however the converse is not true. This fact was observed in \S 11 of \cite{Sholander1952}, which presents the median operation of the four element Boolean algebra as a counter-example.

In this paper, we investigate conservative median algebras and homomorphisms between them, i.e., mappings $f\colon \alg{A}\to\alg{B}$ that are solutions of the functional equation
\begin{equation}
f(\me(x,y,z)) \ = \ \me(f(x), f(y), f(z)).
\end{equation}
We describe such homomorphisms between conservative median algebras $\alg{A}$ and $\alg{B}$. To do so, we present a description of conservative median algebras and semilattices in terms of forbidden substructures (in complete analogy with \textsc{Birkhoff}'s characterization of distributive lattices with $M_5$ and $N_5$ as forbidden substructures),
and that leads to
 a representation of conservative median algebras (with at least five elements) as  median algebras of chains. In fact, the only conservative median algebra that is not representable as a chain is the median algebra of the four element Boolean algebra.% (Theorem \ref{thm:phj}). 

Throughout the paper we employ the  following notation. For each positive integer $n$, we set $[n]=\{1, \ldots, n\}$. Algebras and topological structures are denoted by bold roman capital letters $\alg{A}, \alg{B}, \alg{X}, \alg{Y}, \ldots$ and their universes by italic roman capital letters $A, B, X, Y,\ldots$
To simplify our presentation, we will keep  the introduction of background to a minimum, and  we will assume that the reader is familiar with the theory of lattices and ordered sets. We refer the reader to \cite{Davey2002,Gratzer2003} for further background.  
%To improve the readability of the paper, we adopt the rather unusual convention that in any distributive lattice the empty set is a prime filter and a prime ideal.

%%%% Characterization of conservative blalbkaba
\section{Characterizations of conservative median algebras}
\label{sec:ConsMed}

According to Theorem \ref{thm:semi}, a semilattice can fail to be a conservative median semilattice in tree different ways. First, it can contain a principal ideal which is not a distributive lattice, as in Fig.\ \ref{fig:lat02_00} that depicts 
the bounded lattice $N_5$  that is  not distributive. Second, it can contain three elements $b,c,d$ that do not have a join even though every pair of them is bounded above, such as in Fig.\  \ref{fig:lat02_00bis}. Finally, it can be a median semilattice that is not  conservative, like $\alg{A}_2$  in Fig.\ \ref{fig:bcv} in which $\me_\leq (a,c,d)=b$, and like in Fig.\ \ref{fig:lat04_00}--\ref{fig:lat03_01} in which the semilattices contain a copy of $\alg{A}_2$. Hence, we have proved the following lemma.

\begin{lemma}\label{lem:spo}
The partially ordered sets $\alg{A}_1, \ldots, \alg{A}_5$ depicted in Fig.\ \ref{fig:lat} are not conservative median semilattices. 
\end{lemma}

\begin{figure}
\begin{center}
\subfigure[$\alg{A}_1$]{\label{fig:lat02_00} \begin{tikzpicture}[scale=\tkzscl]
\draw (-0.5,0.5) node{$\bullet$}node[left]{$d'$};
\draw (0,0) node{$\bullet$} node[below]{$a$}--(-1,1)node{$\bullet$} node[left]{$b$};
\draw (0,0)--(1,1) node{$\bullet$} node[right]{$c$};
\draw (1,1)--(0,2) node{$\bullet$} node[above]{$d$};
\draw (-1,1)--(0,2);
\end{tikzpicture}}
\subfigure[$\alg{A}_2$]{\label{fig:bcv}\begin{tikzpicture}[scale=\tkzscl]
\draw (0,0) node{$\bullet$}node[below]{$a$}--(0,1)node{$\bullet$}node[right]{$b$};
\draw (0,1)--(-1,2)node{$\bullet$}node[left]{$c$};
\draw (0,1)--(1,2)node{$\bullet$}node[right]{$d$};
\end{tikzpicture}}\\
\subfigure[$\alg{A}_3$]{\label{fig:lat04_00}\begin{tikzpicture}[scale=\tkzscl]
\draw (-1,1)--(-1,2) node{$\bullet$}node[above]{$d'$};
\draw (0,0) node{$\bullet$} node[below]{$a$}--(-1,1)node{$\bullet$} node[left]{$b$};
\draw (0,0)--(1,1) node{$\bullet$} node[right]{$c$};
\draw (1,1)--(0,2) node{$\bullet$} node[above]{$d$};
\draw (-1,1)--(0,2);
\end{tikzpicture}}
\subfigure[$\alg{A}_4$]{\label{fig:lat03_01} \begin{tikzpicture}[scale=\tkzscl]
\draw (0,2)--(0,2.75) node{$\bullet$}node[right]{$d'$};
\draw (0,0) node{$\bullet$} node[below]{$a$}--(-1,1)node{$\bullet$} node[left]{$b$};
\draw (0,0)--(1,1) node{$\bullet$} node[right]{$c$};
\draw (1,1)--(0,2) node{$\bullet$} node[right]{$d$};
\draw (-1,1)--(0,2);
\end{tikzpicture}}
\subfigure[$\alg{A}_5$]{\label{fig:lat02_00bis} \begin{tikzpicture}[scale=\tkzscl]
\draw (0,0) node{$\bullet$} node[below]{$a$}--(-1,1)node{$\bullet$} node[left]{$b$};
\draw (-1,1) -- (-1,2) node{$\bullet$};
\draw (1,1) -- (1,2) node{$\bullet$};
\draw (0,1) node{$\bullet$} node[right]{c}-- (-1,2);
\draw (0,0)  -- (0,1) -- (1,2);
\draw (0,0)--(1,1) node{$\bullet$} node[right]{$d$};
\draw (1,1)--(0,2) node{$\bullet$};
\draw (-1,1)--(0,2);
\end{tikzpicture}}
\end{center}
\caption{Examples of $\wedge$-semilattices that are not conservative.}\label{fig:lat}
\end{figure}
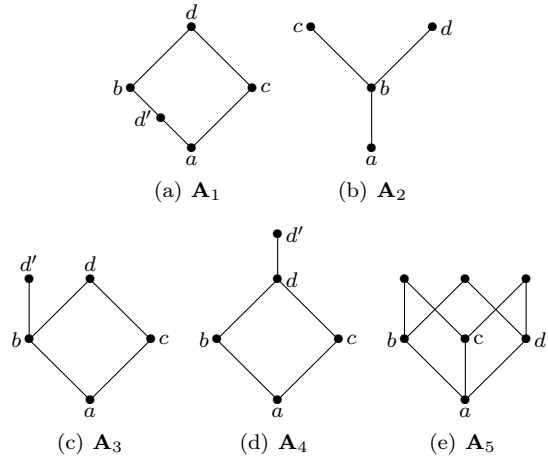

The following theorem provides a description of conservative median algebras and semilattices in terms of forbidden substructures. 

\begin{lemma}
The variety of median algebras satisfies the following equations.
\begin{gather}
\me(x,y,z) = \me\big(\me\big(\me(x,y,z),x,t\big),\me\big(\me(x,y,z),z,t\big),\me\big(\me(x,y,z),y,t\big) \big).\label{eqn:the}
\\
\me(x,y,\me(x,y,z))  = \me(x,y,z),\label{eqn:prem}
\end{gather}
\end{lemma}
\begin{proof}
Every median algebra is isomorphic to a subalgebra of a power of the median algebra $\alg{2}=\struc{\{0,1\}, \me}$, where $\me$ is the majority ternary operation on $\{0,1\}$ (see \cite[Theorem 1.5]{Bandelt1983}). Moreover, equations \eqref{eqn:prem} and \eqref{eqn:the} are satisfied in $\alg{2}$.
\end{proof}

\begin{theorem}\label{prop:fgt}
For every median algebra $\alg{A}$, the following conditions are equivalent.
\begin{enumerate}
\item\label{it:pollo01} $\alg{A}$ is conservative.
\item \label{it:polnewbis} $\alg{A}$ does not contain the median algebra $\alg{A}_2$ depicted in Fig.\ \ref{fig:bcv} as a subalgebra. 
\item\label{it:polnew} For every $a\in A$, the median semilattice $\struc{A, \leq_a}$ does not contain a copy of the poset depicted in Fig.\  \ref{fig:bcv}.
%\item\label{it:pol02} There is an $a \in A$ and lower bounded chains $\alg{C}_0$ and $ \alg{C}_1$  such that $\struc{A, \leq_a}$ is isomorphic to $\alg{C}_0 \bot \alg{C}_1$.
%\item\label{it:pol03} For every $a \in A$, there are lower bounded chains $\alg{C}_0$ and $ \alg{C}_1$  such that $\struc{A, \leq_a}$ is isomorphic to $\alg{C}_0 \bot \alg{C}_1$.
\end{enumerate} 
\end{theorem}
\begin{proof}
First note that any median semilattice with at most four elements is conservative, with the exception of the poset depicted in Fig.\ \ref{fig:bcv}. Hence, we assume that $|A|\geq 5$.

(\ref{it:polnewbis}) $\iff$ (\ref{it:polnew}): Follows directly from the definition of $\leq_a$.

(\ref{it:pollo01})  $\implies$ (\ref{it:polnew}):  Follows from Lemma  \ref{lem:spo}.

(\ref{it:polnew}) $\implies$(\ref{it:pollo01}):  Suppose that $\alg{A}$ is not conservative, that is, there are $a,b,c,d\in A$ such that $d:=\me(a,b,c)\not\in\{a,b,c\}$. Clearly, $a$, $b$ and $c$ must be pairwise distinct. By \eqref{eqn:prem}, $a$ and $b$ are  $\leq_c$-incomparable, and $d<_c a$ and $d<_c b$. Moreover, $c<_c d$ and thus $\struc{\{a, b, c, d\}, \leq_c}$ is a copy of $\alg{A}_2$ in $\struc{A, \leq_c}$.   
\end{proof}

Let $\alg{C}_0=\struc{C_0, \leq_0, c_0}$ and $\alg{C}_1=\struc{C_1, \leq_1, c_1}$ be chains with bottom elements $c_0$ and $c_1$. The  \emph{$\bot$-coalesced sum $\alg{C}_0\bot \alg{C}_1$} of $\alg{C}_0$  and $\alg{C}_1$ is   the poset obtained by amalgamating $c_0$ and $c_1$ in the disjoint union of $C_0$ and $C_1$. Formally, 
\[\alg{C}_0 \bot \alg{C}_1=\struc{C_0\sqcup C_{1}\ /\! \equiv, \ \leq},\]
where $\sqcup$ is the disjoint union, where $\equiv$ is the equivalence generated by $\{(c_0, c_1)\}$ and where $\leq$ is defined by
\[
x/\!\equiv \ \leq \ y/ \!\equiv \ \iff \ (x \in  \{c_0, c_1\} \mbox{ or } x\leq_0 y \mbox{ or } x\leq_1 y).
\]

Theorem \ref{thm:phj} below provides descriptions of conservative median algebras and semilattices  by means of representations by chains. Its proof requires the next technical result. 
%\todo[inline]{In the nextproofs, $x$ has been substituted by $b$, $y$ by $c$, $z$ by $d$ and $z'$ by $d$'. Check with the previous version of the proof if no mistake has been introduced.}

\begin{lemma}\label{lem:dbkf}
For every median algebra $\alg{A}$ with $|A|\geq 5$, the following conditions are equivalent.
\begin{enumerate}
\item\label{it:pol01} $\alg{A}$ is conservative
\item\label{it:pol02} There is an $a \in A$ and lower bounded chains $\alg{C}_0$ and $ \alg{C}_1$  such that $\struc{A, \leq_a}$ is isomorphic to $\alg{C}_0 \bot \alg{C}_1$.
\item\label{it:pol03} For every $a \in A$, there are lower bounded chains $\alg{C}_0$ and $ \alg{C}_1$  such that $\struc{A, \leq_a}$ is isomorphic to $\alg{C}_0 \bot \alg{C}_1$.
\end{enumerate} 
\end{lemma}
\begin{proof}
(\ref{it:pol01}) $\implies$ (\ref{it:pol03}): Let $a \in A$.  First, suppose that for every $b, c \in A\setminus\{a\}$ we have $\me(b,c,a)\not=a$. Since $\alg{A}$ is conservative, for every $b,c\in A$, either $b\leq_a c$ or $c\leq_a b$. Thus $\leq_a$ is a chain with bottom element $a$, and we can choose $\alg{C}_1=\struc{A, \leq_a, a}$ and $\alg{C}_2=\struc{\{a\}, \leq_a, a}$.  

Suppose now that there are $b,c \in A\setminus\{a\}$ such that $\me(b,c, a)=a$, that is,  $b \wedge c=a$. We show that for every $a\in A$,
\begin{equation}\label{eqn:vyt}
d\neq a \implies \big(\me(b,d,a)\neq a \quad \mbox{ or } \quad \me(c,d,a)\neq a\big). %, \quad d \in A.
\end{equation}
%for any $z\neq a$ either $(x,z,a)\neq a$ or $(y,z,a)\neq a$, which exactly means that $x\leq_a z$, $z\leq_a x$, $y\leq_a z$ or $z\leq_a y$. 
For the sake of a contradiction, suppose that  $\me(b,d,a)=a$ and $\me(c,d,a)= a$ for some $d\neq a$. By equation (\ref{eqn:the}), we have
\begin{equation}\label{eqn:the01}
\me(b,c,d)=\me\big(\me\big(\me(b,c,d),b,a\big),\me\big(\me(b,c,d),d,a\big),\me\big(\me(b,c,d),c,a\big) \big).
\end{equation}
Assume that $\me(b,c,d)=b$. Then  (\ref{eqn:the01}) is equivalent to 
$$b=\me(b,\me(b,d,a),\me(b,c,a))=a,$$  
which yields the desired contradiction. By symmetry, we derive the same contradiction in the case $\me(b,c,d)\in \{c,d\}$.

We now prove that for every $a\in A$,
\begin{equation}\label{eqn:vyt02}
d\neq a \implies \big( \me(b,d,a)=a \quad \mbox{ or } \quad  \me(c,d,a)= a\big).%, \quad d \in A.
\end{equation}
For the sake of a contradiction, suppose  that $\me(b,d,a)\neq a$ and $\me(c,d,a)\neq a$ for some $d\neq a$. Since $\me(b,c,a)=a$ we have that $d\not\in \{b,c\}$. 

If $\me(b,d,a)=d$ and $\me(c,d,a)=c$, then $c\leq_a d \leq_a b$ which contradicts $b\wedge c=a$.  Similarly, if $\me(b,d,a)=d$ and $\me(c,d,a)=d$, then $d\leq_a b$ and $d \leq_a c$ which also contradicts  $b \wedge c=a$. The case $\me(b,d,a)=b$ and $\me(c,d,a)=d$ leads to a similar contradiction. 

Hence $\me(b,d,a)=b$ and $\me(c,d,a)=c$, and the $\leq_a$-median semilattice arising from the subalgebra $\alg{B}=\{a,b,c,d\}$ of $\alg{A}$ is  the median semilattice associated with the four element Boolean algebra.
Let $d'\in A \setminus\{a,b,c,d\}$. By (\ref{eqn:vyt}) and symmetry we may assume that $\me(b,d',a)\in \{b,d'\}$.  First, suppose that $\me(b,d',a)=d'$. Then  $\struc{\{a, b, c, d, d'\}, \leq_a}$ is $N_5$ (Fig.\ \ref{fig:lat02_00}) which is not a median semilattice.  
Suppose then  that $\me(b,d',a)=b$. 
In this case, the restriction of $\leq_a$ to $\{a, b, c, d, d'\}$ is depicted in Fig.\ \ref{fig:lat04_00} or \ref{fig:lat03_01}, which contradicts Proposition  \ref{lem:spo}, and the proof of (\ref{eqn:vyt02}) is thus complete.

Now, let ${C}_0=\{d \in A \mid (b,d,a)\neq a\}$, ${C}_1=\{d \in A \mid (c,d,a)\neq a\}$ and let $\alg{C}_0=\struc{C_0, \leq_a, a}$ and $\alg{C}_1=\struc{C_1, \leq_a, a}$. It follows from (\ref{eqn:vyt}) and (\ref{eqn:vyt02}) that $\struc{\alg{A}, \leq_a}$ is isomorphic to $\alg{C}_0 \bot \alg{C}_1$.

(\ref{it:pol03}) $\implies$ (\ref{it:pol02}): Trivial.

(\ref{it:pol02}) $\implies$ (\ref{it:pol01}): Let $b,c,d\in \alg{C}_0 \bot \alg{C}_1$. If $b,c,d \in C_i$ for some $i \in \{0,1\}$ then $\me(b,c,d)\in \{b,c,d\}$. Otherwise, if  $b, c \in C_i$ and $d \not\in C_{i}$, then $\me(b,c,d)\in \{b,c\}$.  
\end{proof}

\begin{theorem}\label{thm:phj}
Let $\alg{A}=\struc{A, \me}$ be a median algebra with $|A|\geq 5$. Then $\alg{A}$ is conservative if and only if there is a total order $\leq$ on $A$ such that  $\me=\me_\leq$.

Consequently, if $\alg{A}$ is a conservative median algebra whose  operation is not the median operation of a totally ordered set, then $\alg{A}$ is isomorphic to $\alg{2}\times \alg{2}$.
\end{theorem}
\begin{proof}
We have already noted that if $\leq$ is a total order on $A$ then $\struc{A,\me_\leq}$ is conservative. Now assume that $\alg{A}=\struc{A, \me}$ is a conservative median algebra with $|A|\geq 5$. Consider the universe of $\alg{C}_0\bot\alg{C}_1$  in condition (\ref{it:pol02}) of Lemma \ref{lem:dbkf} endowed with $\leq$ defined by $x\leq y$ if $x\in C_1$ and $y \in C_0$ or $x, y \in C_0$ and $x\leq_0 y$ or $x,y \in C_1$ and $y\leq_1 x$. Clearly, $\leq$ is a  total order and $\me_\leq=\me$.

For the second part of the proof, note that the only conservative median algebra with at most four elements whose median operation is not the median operation of a totally ordered set is $\alg{2}\times \alg{2}$.
\end{proof}
\section{Duality theory toolbox}\label{sec:prel}

In this section, we recall a  dual equivalence between the category of median algebras and a category of structured topological spaces. It was first exposed in \cite{Isbell1980} and was stated in terms of homomorphisms into the median algebra $\alg{2}$.  It was later on recognized \cite{Werner1981,Davey1983} as being an instance of a general scheme of dualities for finitely generated quasi-varieties of algebras known as \emph{natural duality} \cite{NDFWA}.  This general approach, as well as its application to the variety of median algebras, is fully exposed in \cite[Section 4.3]{NDFWA}.  
%The sets $u^{-1}(0)$ where $u\colon \alg{A}\to \alg{2}$ is a median homomorphism can be axiomatized as follows.

%Since a median homomorphism $u\colon \alg{A}\to \alg{2}$ is completely characterized by $u^{-1}(0)$, this dual equivalence can also be stated in terms of prime convex subsets.  It is this language that we use to recall the duality.

\begin {definition}[\cite{Bandelt1983}]
Let $\alg{A}=\struc{A, \me}$ be a median algebra. A subset $C$ of $A$ is \emph{convex} if $\me(c_1,c_2,a)\in C$ whenever $c_1, c_2 \in C$ and $a \in A$. A convex subset $C$ of $A$ is \emph{prime} if its complement $A \setminus C$  in $A$ is also convex. We denote by $\spec{ \alg{A}}$ the set of prime convex subsets of the median algebra $\alg{A}$. 
\end {definition}
Equivalently, $C\subseteq A$ is  a prime convex  subset if it satisfies the following condition: for every $x, y, z \in A$, the element $\me(x,y,z)$ belongs to $C$ if and only if at least one of the sets $\{x,y\}, \{x,z\}, \{y,z\}$ is a subset of $C$. 

\begin {proposition}[Proposition 1.3 in \cite{Bandelt1983}]\label{lem:fgt}
If\/ $\alg{L}$ is a bounded distributive lattice, then the prime convex subsets of\/ $\alg{L}$ are its prime filters and prime ideals.
\end {proposition}

It is not difficult to check that in a median algebra $\alg{A}$, prime convex subsets coincide with the sets $u^{-1}(0)$ where $u\colon \alg{A}\to \alg{2}$ is a median homomorphism. It is convenient to use  prime convex subsets instead of homomorphisms $u\colon \alg{A}\to \alg{2}$ in the dual equivalence we use in this paper.
As noted in \cite{Isbell1980,Werner1981}, the set $\spec{\alg{A}}$  can be equipped with a topological structure that completely characterizes  $\alg{A}$. We recall this construction in the remainder of this section. For $a \in \alg{A}$ we denote by $r_a$ the set $\{I \in \spec{\alg{A}}\mid a \not\in I\}$.

\begin {definition}
Let $\alg{A}$ be a median agebra.  The \emph{dual $\alg{A}^*$ of $\alg{A}$} is the topological structure $\alg{A}^*=\struc{\spec{\alg{A}}, \subseteq, \cdot^c, \varnothing, A, \tau}$ where  $\cdot^c$ is the set-complement  in A and $\tau$ is the topology with subbasis  $\{r_a \mid a \in A\} \cup \{\spec{\alg{A}}\setminus r_a \mid a \in A\}$. 

Furthermore, for a homomorphism $f:\alg{A} \to \alg{B}$ between median algebras, let $f^*$ the map defined on $\alg{B}^*$ by $f^*(I)=f^{-1}(I)$. It is not difficult to check using the definitions that $f^*$ is valued in $\alg{A}^*$
\end {definition}

\begin {remark}By the \emph{Prime Convex Theorem} \cite[Theorem 13]{Nieminen1978}, it follows that $\{r_a \mid a \in A\} \cup \{\spec{A}\setminus r_a \mid a \in A\}$ is in fact a basis. 
%Equivalently, this result can be viewed as a consequence of the \emph{Separation Theorem}  \cite[Chapter 1, Theorem 4.4]{NDFWA} since $\var{M}$ coincides with $\classop{ISP}(\alg{2})$ where $\alg{2}$ is the two element median algebra (in which the median operation is the majority term). 
%In particular, to check that an equation holds in every median algebra, it suffices to prove that it holds in $\alg{2}$. For example, the equation
%\begin{equation}\label{eqn:the}
%\me(x,y,z)=\me\big(\me\big(\me(x,y,z),x,t\big),\me\big(\me(x,y,z),z,t\big),\me\big(\me(x,y,z),y,t\big) \big)
%\end{equation}
%is satisfied in $ \var{M}$.
\end {remark}

The class of duals of median algebras can be defined as follows.

\begin {definition}[\cite{NDFWA}]
A \emph{bounded strongly complemented \textsc{Priestley} space} is a topological structure
$\topo{X}=\struc{X, \leq, \cdot^c, 0, 1, \tau}$ where $\struc{X,\leq,\tau}$ is a \textsc{Priestley} space  with 0 and 1 as bottom and top elements, and $\cdot^c$ is an order reversing homeomorphism  that satisfies 
\[x\leq x^c \implies x=0 \quad \text{and} \quad x^{cc}=x.\] 
\end {definition}

Bounded strongly complemented \textsc{Priestley} spaces are called \emph{bounded totally ordered disconnected compact spaces with an involution} in \cite{Werner1981}.

\begin {definition}
A \emph{complete ideal $W$}  of a bounded strongly complemented \textsc{Priestley} space  $\topo{X}$ is a clopen downset that satisfies   $x \in W$ if and only if $x^c \not\in W$.
% In particular, a complete ideal is never empty. 
With no danger of ambiguity, we also denote the set of complete ideals of $\topo{X}$ by $\spec{\topo{X}}$. This set is turned into the algebra $\topo{X}_*=\struc{\spec{\topo{X}}, \me}$ where $\me$ is the restriction of  $\me_{2^{X}}$ to $\spec{\topo{X}}$. 
For a continuous structure-preserving  map $\phi\colon \topo{X} \to \topo{Y}$, we define $\phi_*$  to be the map on $\topo{Y}_*$ given by $\phi_*(W)=\phi^{-1}(W)$.
\end {definition}

The class  $\class{X}$ of bounded strongly complemented \textsc{Priestley} spaces can be thought of as a category with continuous structure-preserving maps as arrows. Likewise, the variety $\var{M}$ of median algebras is thought of as a category with homomorphisms as arrows.    For $\topo{X}, \topo{Y}\in \class{X}$, we say that $\topo{Y}$ is a \emph{substructure} of $\topo{X}$ if $Y$ is a closed subset of $\struc{X,\tau}$ and $\topo{Y}$ is induced by the restriction of $\topo{X}$ to $Y$. In that case, if $\psi\colon \topo{Z}\to \topo{Y}$ is an isomorphism, we say that $\psi$ is an embedding of $\topo{Z}$ into $\topo{X}$.

%For any $\alg{A}\in \var{M}$ and any $\topo{X}\in \class{X}$, we define the maps $r_{\alg{A}}:\alg{A} \to 2^{\alg{A}^*}$ and $\rho_{\topo{X}}:\topo{X} \to 2^{\topo{X}_*}$ by
%\[
%r_{\alg{A}}(a)=\{I \in \alg{A}^* \mid a \not\in I\},  
%\] 
%\[
%\rho_{\topo{X}}(I)=\{x \in \topo{X}_* \mid I \not\in x\}.
%\]
%The maps $r_{\alg{A}}$ and $\rho_{\topo{X}}$ are embeddings and $(\cdot^*, \cdot_*, r, \rho)$ is a dual adjunction between $\var{M}$ and $\class{X}$. Actually, this connection can be turned into a dual equivalence.

\begin {proposition}[\cite{NDFWA,Isbell1980,Werner1981}] \label{prop:fsg}
The functors $\cdot^*:\var{M} \to \class{X}$ and $\cdot_*:\class{X} \to \var{M}$ define a dual equivalence between the categories $\var{M}$ and $\class{X}$.
%Moreover, they turn embeddings into onto morphisms, and onto morphisms into embeddings. Hence, if $\alg{A}, \alg{B}\in \var{M}$ then $(\alg{A}\times \alg{B})^*\cong \alg{A}^* \oplus\alg{B}^*$.
\end {proposition}

\begin {remark}
The isomorphism between $\alg{A}$ and $(\alg{A}^*)_*$ mentioned in Proposition \ref{prop:fsg} is given by $a\mapsto r_a$.
\end {remark}

We denote by  $\topo{X} \oplus \topo{Y}$ the coproduct of $\topo{X},\topo{Y}\in \class{X}$. It is not difficult to check that $\topo{X} \oplus \topo{Y}$ is realized in $\class{X}$ by amalgamating $0$ and $1$ of $\topo{X}$ with $0$ and $1$ of $\topo{Y}$, respectively,  in the disjoint union of $\topo{X}$ and $\topo{Y}$.

It is a general result of category theory that under a dual equivalence, products in one category correspond to coproducts in the other category  (for instance, see \cite[Chapter 1, Lemma 1.4]{NDFWA}). In particular, we have
\[
(\alg{A}\times \alg{B})^* \cong \alg{A}^* \oplus \alg{B}^*,
\]
for every $\alg{A}, \alg{B} \in \var{M}$. Moreover, we have the following useful result.
\begin {proposition}\label{prop:inj}
A homomorphism $f\colon \alg{A}\to \alg{B}$ between two median algebras $\alg{A}$ and $\alg{B}$ is onto if and only if $f^*\colon \alg{B}^* \to \alg{A}^*$ is an embedding.
\end {proposition}
\begin{proof}
%Assume that $f$ is an onto homomorphism and let $I, J \in \alg{B}^*$ such that $f^*(I)=F^*(J)$, that is, $f^{-1}(I)=f^{-1}(J)$. For the sake of contradiction, suppose that there is an $i\in I\setminus J$ and let $a\in I$ such that $f(a)=i$. We obtain that $a\in f^*(I)\setminus f^*(J)$, a contradiction.

Stated in the language of Natural Duality, %(prime convex subsets of a median algebra $\alg{A}$ correspond to the sets $u^{-1}(0)$ for homomorphisms $u\colon \alg{A}\to \alg{2}$), 
the dual equivalence of Proposition \ref{prop:fsg} is a strong duality (see \cite[Chapter 4, Therorem 3.4]{NDFWA}). Then, the proof is an application of \cite[Chapter 3, Lemma 2.6]{NDFWA}.
\end{proof}

\section{Homomorphisms between conservative median algebras}
We now use the duality theory apparatus  recalled in Section \ref{sec:prel} to describe median homomorphisms between  (products of) conservative median algebras.

First, we characterize the duals of the conservative median algebras. Let $\alg{P}_0=\struc{P_0, \leq_0, 0_0, 1_0}$ and $\alg{P}_1=\struc{P_1, \leq_1, 0_1, 1_1}$ be two bounded posets. As in Section \ref{sec:ConsMed},    $\alg{P}_0 \amalg \alg{P}_1$ denotes the \emph{coalesced} sum of $\alg{P}_0$ and $\alg{P}_1$, that is, the poset obtained from the disjoint union of $\alg{P}_0$ and $\alg{P}_1$ by identifying $0_0$ with $0_1$, and $1_0$ with $1_1$. We denote by $i_{\alg{P}_k}$ the natural embedding  $i_{\alg{P}_k}: \alg{P}_k \to \alg{P}_0 \amalg \alg{P}_1 $ for $k \in \{0, 1\}$. To simplify notation, we often identify $\alg{P}_k$ with its copy $i_{\alg{P}_k}(\alg{P}_k)$ in $\alg{P}_0 \amalg \alg{P}_1$ for $k \in \{0,1\}$.

If $\struc{\alg{C}, \tau'}$ is a bounded \textsc{Priestley} chain (i.e., a bounded totally ordered \textsc{Priestley} space, see, e.g., \cite{Davey2002}),  $\alg{C}\amalg \alg{C}^\partial$ can be endowed with  an operation $\cdot^c$ and a topology $\tau$, so that  $\struc{\alg{C}\amalg \alg{C}^\partial, \cdot^c, \tau}$ is a bounded strongly complemented \textsc{Priestley} space. Indeed, it suffices  to define 
\begin{itemize}
\item $\tau$ as the final topology relative  to  $i_{\alg{C}}$ and $i_{\alg{C}^\partial}$ (\emph{i.e.}, the finest topology that makes $i_{\alg{C}}$ and $i_{\alg{C}^\partial}$ continuous),
\item  $\cdot^c$ as the function that maps the bottom element $0$ to the top element $1$ and conversely, and that maps each element of $\alg{C}\setminus\{0,1\}$ to its copy in $\alg{C}^\partial$ and conversely.
\end{itemize}
With no danger of ambiguity, we use  $\alg{C}\amalg \alg{C}^\partial$ to denote  $\struc{\alg{C}\amalg \alg{C}^\partial, \cdot^c, \tau}$.

For $\topo{X}\in \class{X}$ and $Y \subseteq X$ set $Y^c=\{x^c \in X\mid x \in Y\}$. Also, for a \textsc{Priestley} space $\struc{\alg{P},\tau}$, let $\Clo{\struc{\alg{P},\tau}}$ be the set of its nonempty proper  clopen downsets ordered by inclusion. Moreover, for a poset $\alg{P}$, let $\struc{\Upb{\alg{P}},\tau}$ be the set of its upsets ordered by inclusion and equipped with the topology $\tau$ which has $\{\{I \in \Upb{\alg{P}}\mid p \not\in I \}\mid p \in P\}\cup \{\{I \in \Upb{\alg{P}}\mid p \in I \}\mid p \in P\}$ as subbasis. If $\alg{P}$ is a chain, then $\struc{\Upb{\alg{P}},\tau}$ is a bounded \textsc{Priestley} space.
\begin {proposition}\label{prop:nzf}
Let $\alg{A}=\struc{A,\me}$ be a median algebra with $|A|\geq 5$. The following conditions are equivalent. 
\begin{enumerate}
\item\label{it:poi01} $\alg{A}$ is conservative.
\item\label{it:poi02} There is a bounded \textsc{Priestley} chain $\struc{\alg{C}, \tau}$ such that $\alg{A}^*$ is isomorphic to $\alg{C} \amalg \alg{C}^\partial$.
\item\label{it:poi03} $\alg{A}$ is the median algebra of  the nonempty proper clopen downsets of a bounded \textsc{Priestley} chain $\struc{\alg{C},\tau}$.
\end{enumerate}
Furthermore, if one of these conditions is satisfied and if $\alg{C}_0=\struc{A,\leq}$ is a chain representation of $\alg{A}$ given by Theorem  \ref{thm:phj},  then $\alg{A}^* \cong \Upb{\alg{C}_0}\amalg\Upb{\alg{C}_0}^\partial$ and $\me$ is the median operation of $\Clo{\Upb{\alg{C}_0}}$.
\end {proposition}
\begin{proof}
(\ref{it:poi01}) $\implies$ (\ref{it:poi02}): According to Theorem \ref{thm:phj}, there is a totally ordered set $\alg{C}_0=\struc{A, \leq}$   such that $\alg{A}=\struc{A, \me_{\leq}}$.   From Proposition \ref{lem:fgt}, we know that the prime convex subsets of $\struc{A, \me_\leq}$ are the prime filters and prime ideals of $\alg{C}_0$, that is, the upsets of $\alg{C}_0$ and  the downsets of $\alg{C}_0$.  Then  $\alg{A}^*$ is isomorphic to  $\Upb{\alg{C}_0}\amalg \Upb{\alg{C}_0}^\partial$.

(\ref{it:poi02}) $\implies$ (\ref{it:poi03}): The median algebra $\alg{A}$ is isomorphic to  $(\alg{C} \amalg \alg{C}^\partial)_*$. If $W$ is a complete ideal of $\alg{C} \amalg \alg{C}^\partial$ 
then $\omega_W:=W\cap  \alg{C}$ belongs to $\Clo{\struc{\alg{C}, \tau}}$. Conversely, if $\omega\in \Clo{\struc{\alg{C}, \tau}}$ then $W_\omega:=\omega \cup (\alg{C}^\partial \setminus \omega^c)$ is a complete ideal of $\alg{C} \amalg \alg{C}^\partial$. It is not difficult to check that the maps $\omega_{-}\colon(\alg{C} \amalg \alg{C}^\partial)_* \to \Clo{\struc{\alg{C}, \tau}}$ and $W_{-}\colon \Clo{\struc{\alg{C}, \tau}} \to (\alg{C} \amalg \alg{C}^\partial)_*$ are median homomorphisms such that one is the inverse of the other. We conclude that up to isomorphism, $\me$ is the median operation of $\Clo{\struc{\alg{C}, \tau}}$.

(\ref{it:poi03}) $\implies$ (\ref{it:poi01}): Follows  straightforwardly since $\me$ is the median operation of a chain.

The proof of the first and the second claims of the last statement are given in the proof of (\ref{it:poi01}) $\implies$ (\ref{it:poi02}) and  (\ref{it:poi02}) $\implies$ (\ref{it:poi03}), respectively.
\end{proof}

%From Proposition \ref{prop:nzf} it follows that the totally ordered set $\alg{C}$ given in Theorem  \ref{thm:phj} is unique, up to (dual) isomorphism.

\begin {corollary}\label{cor:grs}
Let $\alg{A}$ be a median algebra. If $\alg{C}$ and $\alg{C}'$ are two chains such that $\alg{A}\cong \struc{\alg{C}, \me_{\alg{C}}}$ and 
$\alg{A}\cong \struc{\alg{C}', \me_{\alg{C}'}}$, then $\alg{C}$ is order isomorphic or dual order isomorphic to $\alg{C}'$.
\end {corollary}

Given a conservative median algebra $\alg{A}=\struc{A, \me}$ (with $|A|\geq 5$), Theorem \ref{thm:phj} provides with a total order $\leq_\alg{A}$ on $A$  such that $\me=\me_{\leq_{\alg{A}}}$. Corollary \ref{cor:grs} states that $\struc{A,\leq_{\alg{A}}}$ is unique up to isomorphisms and dual isomorphisms. We call $\leq_{\alg{A}}$ the \emph{chain ordering of $\alg{A}$} and we denote $\struc{A, \leq_\alg{A}}$ by $\alg{C}(\alg{A})$  .

%Hence, if $\alg{A}$ is a conservative median algebra with at least five elements, let  $\alg{C}(\alg{A})$ be the chain (unique up to order-isomorphism) such that $\alg{A}\cong \struc{\alg{C}(\alg{A}), \me_\alg{C(\alg{A})}}$ and let $f_{\alg{A}}:\alg{A} \to \struc{\alg{C}(\alg{A}), \me_\alg{C(\alg{A})}}$ be the corresponding isomorphism.  If $f:\alg{A} \to \alg{B}$ is a map between two median algebras, the map $f':\alg{C}(\alg{A})  \to \alg{C}(\alg{B})$ defined as $f'=f_{\alg{B}} \circ f \circ f_{\alg{A}}^{-1}$ is said to be \emph{induced by $f$}.

We use Proposition \ref{prop:nzf} to characterize median homomorphisms between conservative median algebras. Recall that a {map} between two posets is \emph{monotone} if it is isotone or antitone.

\begin {proposition}
Let $\alg{A}$ and $\alg{B}$ be two conservative median algebras with at least five elements. A map $f:\alg{A}\to \alg{B}$  is a median homomorphism if and only if it is monotone with respect to the chain orderings of $\alg{A}$ and $\alg{B}$. 
%$f\colon \alg{C}(\alg{A})\to \alg{C}(\alg{A}) $ is  monotone.
\end {proposition}
\begin{proof}
(Necessity) We may assume that $f$ is onto. According to Proposition \ref{prop:inj}, the map \[f^*: \Upb{\alg{C}(\alg{B})} \amalg \Upb{\alg{C}(\alg{B})}^\partial \hookrightarrow \Upb{\alg{C}(\alg{A})} \amalg \Upb{\alg{C}(\alg{A})}^\partial\] is a $\class{X}$-embedding.

If the range of $f^*$ is equal to $\{0,1\}$, then $\alg{B}$ is the one-element median algebra and $\alg{C}(\alg{B})$ is the one-element chain, and the result follows trivially. Hence, we may assume that there is a $I\in \Upb{\alg{C}(\alg{B})}$ such that $f^*(I)\not\in\{0,1\}$.
If $f^*(I) \in \Upb{\alg{C}(\alg{A})}$, then  $f^*(\Upb{\alg{C}(\alg{B})})\subseteq  \Upb{\alg{C}(\alg{A})}$ since $f^*$ is isotone. We prove that $f\colon \alg{C}(\alg{A})\to \alg{C}(\alg{B})$ is isotone. Suppose that $a\leq b$ for some $a,b\in \alg{C}(\alg{A})$.  Then $f^*\big([f(a))\big)$ contains $b$ since it is an upset that contains $a$ and $a\leq b$. It means that $f(a)\leq f(b)$, which is the desired result.

If $f^*(I) \in \Upb{\alg{C}(\alg{A})}^\partial$, we conclude in a similar way that $f\colon \alg{C}(\alg{A})\to \alg{C}(\alg{B})$ is antitone.

(Sufficiency)  If $f\colon \alg{C}(\alg{A})\to \alg{C}(\alg{B})$ is isotone, then it maps upsets to upsets and downsets to downsets. If it is antitone, it maps upsets to downsets and conversely. It means that $f^*$ is valued in $\Upb{\alg{C}(\alg{A})} \amalg \Upb{\alg{C}(\alg{A})}^\partial$. It is then straightforwad to check that $f^*$ is a $\class{X}$-morphism.    
\end{proof}

\begin {corollary}\label{cor:mpo}
Let $\alg{C}$ and $\alg{C}'$ be two chains. A map $f:\alg{C}\to \alg{C}'$ is a median homomorphism if and only if it is monotone.
\end {corollary}
\begin {remark}\label{rem:vcf}
Note that Corollary \ref{cor:mpo} only holds for chains. Indeed, Fig.\ \ref{fig:lop01}  gives  an example of a monotone map that is not a median homomorphism, and Fig.\ \ref{fig:lop02}  gives an example of median homomorphism that is not monotone.
\begin{figure}
\hspace*{\fill}
\subfigure[A monotone map which is not a median homomorphism.]{\label{fig:lop01}
\begin{tikzpicture}[scale=\tkzscl]
\draw (0,0) node{$\bullet$}--(-1,1)node{$\bullet$}--(0,2)node{$\bullet$} ;
\draw (0,0)--(1,1)node{$\bullet$}--(0,2) ;
\draw (3,0)node{$\bullet$}--(3,2)node{$\bullet$};
\draw[->,>=latex, dashed] (0,0) to[bend left=10] (3,0);
\draw[->,>=latex, dashed] (-1,1) to[bend left=10] (3,2);
\draw[->,>=latex, dashed] (1,1) to[bend right=10] (2.95,1.95);
\draw[->,>=latex, dashed] (0,2) to[bend left=10] (2.95,2.05);
\end{tikzpicture}}
\hfill
\subfigure[A median homomorphism which is not monotone.]{\label{fig:lop02}
\begin{tikzpicture}[scale=\tkzscl]
\draw (0,2)--(0,2.75) node{$\bullet$};
\draw (0,0) node{$\bullet$} --(-1,1)node{$\bullet$} ;
\draw (0,0)--(1,1) node{$\bullet$};
\draw (1,1)--(0,2) node{$\bullet$} ;
\draw (-1,1)--(0,2);
%\draw (3,0)node{$\bullet$}--(3,1)node{$\bullet$}--(3,2)node{$\bullet$};
%
\draw (3,0) node{$\bullet$}--(2,1)node{$\bullet$}--(3,2)node{$\bullet$} ;
\draw (3,0)--(4,1)node{$\bullet$}--(3,2) ;

\draw[->,>=latex, dashed] (0,0) to[bend left=10] (1.95,0.95);
\draw[->,>=latex, dashed] (-1,1) to[bend left=10] (1.95,1.05);
\draw[->,>=latex, dashed] (1,1) to[bend right=20] (2.95, 0.05);
\draw[->,>=latex, dashed] (0,2) to[bend left=40] (3.05, 0.07);
\draw[->,>=latex, dashed] (0,2.75) to[bend left=10] (3.95, 1);
\end{tikzpicture}}
\hspace*{\fill}
\caption{Examples for Remark \ref{rem:vcf}.}
\end{figure}
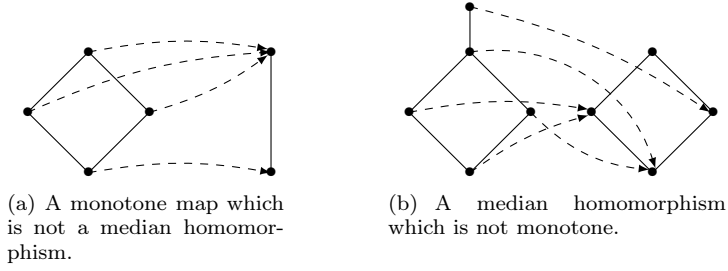
\end {remark}

Since the class of conservative median algebras is  closed under homomorphic images, we obtain the following corollary.

\begin {corollary}
Let $\alg{A}$ and $\alg{B}$ be two median algebras and $f:\alg{A}\to\alg{B}$.  If $\alg{A}$ is conservative, and if $|A|, |f(A)|\geq 5$, then $f$ is a median homomorphism if and only if $f(\alg{A})$ is a conservative  median subalgebra of $\alg{B}$  and $f$ is monotone with respect to the chain orderings of $\alg{A}$ and $f(\alg{A})$.
\end {corollary}

%\todo[inline]{Can we obtain a general result about distributive lattice by considering that it is obtained by 'gluing' chains?}

%Let $\struc{\alg{L}, \me_\alg{L}}$ and $\struc{\alg{L}', \me_{\alg{L}'}}$ be two median algebras associated to two distributive lattices $\alg{L}$ and $\alg{L}'$ and $f:L\to L'$. 

The dual equivalence between $\class{M}$ and $\class{X}$ turns finite products into finite coproducts. This property can be used to characterize median homomorphisms between finite products of chains. If $f_i:A_i \to A'_i$ ($i \in [n]$) is a family of maps, let $(f_1, \ldots, f_n):A_1\times \cdots \times A_n \to A'_1\times \cdots \times A'_n$ be defined by
\[
(f_1, \ldots, f_n)(x_1, \ldots, x_n):=(f_1(x_1), \ldots, f_n(x_n)).
\]

The following proposition essentially states that median homomorphisms between finite products of chains necessarily decompose componentwise.

\begin {proposition}\label{prop:jgq}
Let $\alg{A}=\alg{C}_1\times  \cdots \times \alg{C}_k$ and  $\alg{B}=\alg{D}_1\times \cdots\times \alg{D}_n$ be two finite products of  chains. Then  $f:\alg{A}\to \alg{B}$ is a median homomorphism if and only if there exist $\sigma:[n]\to[k]$ and monotone maps $f_i:\alg{C}_{\sigma(i)}\to \alg{D}_i$ for $i\in [n]$ such that $f=(f_{\sigma(1)}, \ldots, f_{\sigma(n)})$.
\end {proposition}
\begin{proof}
The condition is clearly sufficient. To prove that it is necessary, let $\alg{A}$, $\alg{B}$ and $f$ be as in the statement.  The map $f^*= \alg{D}_1^{*} \oplus\cdots \oplus \alg{D}_n^{*} \to  \alg{C}_1^{*} \oplus\cdots \oplus \alg{C}_k^{*} $ is an $\class{X}$-morphism. Let $i \in [n]$. Since $\alg{D}_i^*$ is a  $\class{X}$-substructure of $\alg{B}^*\cong\alg{D}_1^{*} \oplus\cdots \oplus \alg{D}_n^{*}$, the map $f^*\vert_{\alg{D}_i^*}$ is an $\class{X}$-morphism from $\alg{D}_i^*$ to $\alg{A}^*\cong \alg{C}_1^{*} \oplus\cdots \oplus \alg{C}_k^{*}$. Hence, there is a $\sigma(i)\in[k]$ such that $f^*\vert_{\alg{D}_i^*}$ is valued in $\alg{C}_{\sigma(i)}^{*}$. It follows that the diagram in Fig.~\ref{fig:jgq01} commutes, and by duality, so is the diagram in Fig.~\ref{fig:jgq02}. Hence, it suffices to define $f_{\sigma_i}$ as $(f^*\vert_{\alg{D}_i^*})^*$ to conclude the proof.
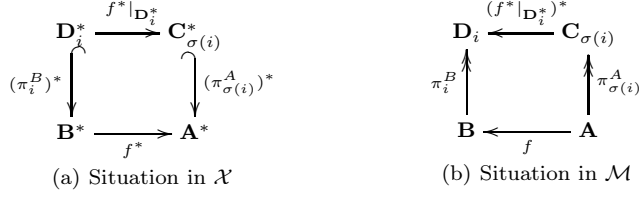
\begin{figure}
\hspace*{\fill}
\subfigure[Situation in $\class{X}$]{\label{fig:jgq01}\xymatrix{\alg{D}_{i}^*\ar[r]^{f^*\vert_{\alg{D}_i^*}}\ar@{^{(}->}[d]_{(\pi^{B}_i)^*} & \alg{C}^*_{\sigma(i)} \ar@{_{(}->}[d]^{(\pi^{A}_{\sigma(i)})^*} \\
\alg{B}^*\ar[r]_{f^*} & \alg{A}^*}}
\hfill
\subfigure[Situation in $\class{M}$]{\label{fig:jgq02}\xymatrix{\alg{D}_i & \alg{C}_{\sigma(i)}\ar[l]_{(f^*\vert_{\alg{D}_i^*})^*}  \\
\alg{B}\ar@{->>}[u]^{\pi^{B}_i} & \alg{A} \ar[l]^{f} \ar@{->>}[u]_{\pi^{A}_{\sigma(i)}} }}
\hspace*{\fill}
\caption{Proof of Proposition \ref{prop:jgq}}
\end{figure}
\end{proof}

%The following corollary is pertaining to aggregation function theory. It essentially states that an aggregation function on ordinal scales is median-preserving if and only if it is dictatorial and monotone.
 If $A=A_1\times \cdots \times A_n$ and $i \in [n]$, then we denote  the  projection map from $A$ onto $A_i$ by $\pi^A_i$, or simply by $\pi_i$ if there is no danger of ambiguity. %\textcolor{red}{Note that any map $f\colon A \to A_1 \times \cdots \times A_n$ valued in a Cartesian product can be considered as a tuple-valued function: $f(\bfa)=(f_1(\bfa), \cdots, f_n(\bfa )$ where $f_i:A \to A_i$ for every $i$}

\begin {corollary}
Let $\alg{C}_1, \ldots, \alg{C}_n$ and $\alg{D}$ be chains. A map $f:\alg{C}_1\times \cdots \times \alg{C}_n\to \alg{D}$ is a median homomorphism if and only if there is a $j\in [n]$ and a monotone map $g:\alg{C}_j \to \alg{D}$ such that $f=g\circ \pi_j$.
\end {corollary}

In the particular case of Boolean algebras, Proposition \ref{prop:jgq} can be restated as in the following corollary. 

\begin {corollary}\label{cor:bool01}
Assume that $f:\alg{2}^n \rightarrow \alg{2}^m$ is a map between two finite Boolean algebras. The map $f$ is a median homomorphism if and only if there are  $\sigma: [m] \to ([n] \cup\{\bot\})$ and $\varepsilon:[m]\to  \{\mathrm{id},\neg\}$ such that \[f: (x_1, \ldots, x_n) \mapsto (\varepsilon_1 x_{\sigma_1}, \ldots, \varepsilon_m x_{\sigma_m}),\]
where $x_\bot$ is defined as the constant map $0$.
\end {corollary}

\begin {corollary} %Assume that $\alg{A}\cong\alg{2}^n$ is a finite Boolean algebra. 
\begin{enumerate}\label{cor:bool02}
\item The Boolean functions on $\alg{2}^n$ that are median homomorphisms are exactly the constant functions, the projection maps $\pi\colon\alg{2}^n \to \alg{2}$ and the negations of the projection maps.
\item\label{it:jhh02} A map $f\colon\alg{2}^n \rightarrow \alg{2}^n$ is a median isomorphism if and only if there is a permutation $\sigma$ of $[n]$ and an element $\varepsilon$ of $\{\mathrm{id}, \neg\}^n$ such that $f(x_1, \ldots, x_n)=(\varepsilon_1 x_{\sigma(1)}, \ldots, \varepsilon_n x_{\sigma(n)})$ for any $(x_1, \ldots, x_n)$ in $\alg{2}^n$.
\end{enumerate}
\begin {remark}
As kindly noticed by the reviewer, Corollaries  \ref{cor:bool01} and \ref{cor:bool02} follow from  properties of congruence distributive varieties generated by a finite simple algebra. For instance, it can be shown that if $\alg{A}$ is a finite simple algebra that generates a congruence distributive variety and if $f\colon \alg{A}^n \to \alg{A}^n$ is an isomorphism, then there exist a permutation $\sigma$ of $[n]$ and  automorphisms $\varepsilon_1, \ldots, \varepsilon_n$ of $\alg{A}$  such that $f(x_1, \ldots, x_n)=(\varepsilon_1 x_{\sigma(1)}, \ldots, \varepsilon_n x_{\sigma(n)})$ for every $(x_1, \ldots, x_n) \in \alg{A}^n$. Since the variety of median algebras has a near-unanimity term, it is congruence distributive (see \cite[Theorem 2]{Mitschke1978}) and hence Corollary \ref{cor:bool02}.\ref{it:jhh02} can be obtained from  the latter result.
\end {remark}

\end {corollary}
\section{Concluding remarks and further research directions}
In this paper we have described conservative median algebras and semilattices with at least five elements in terms of forbidden configurations and have given a representation by chains. We have also characterized median homomorphisms between finite products of these algebras, showing that they are essentially determined componentwise. The next step in this line of research is to extend our results to larger classes of median algebras and their ordered counterparts. The topological duality for the variety of median algebras recalled in this paper may again turn out to be a valuable tool.

Another research direction would be to turn the representation theorem stated in Proposition \ref{prop:nzf} into a dual equivalence, and to use this equivalence to describe existentially and algebraically closed elements in the category of conservative median algebras by following the ideas developed in \cite[Chapter 5]{NDFWA}.

\section*{Acknowledgment}
We would like to thank an anonymous referee for his careful comments which helped to improve the readability of the paper.

This work was supported  by the internal research project F1R-MTHPUL-15MRO3 of the University of Luxembourg.


\begin{thebibliography}{10}

\bibitem{Avann1961}
S. P. Avann, 
\newblock {Metric ternary distributive semi-lattices.} 
\newblock{\em Proceedings of the American Mathematical Society},12:407--414, 1961.

\bibitem{Bandelt1983}
H.~J.~Bandelt and J.~Hedl\'{\i}kov\'{a}.
\newblock {Median algebras}.
\newblock {\em Discrete mathematics}, 45:1--30, 1983.

\bibitem{Birkhoff1947}
G. Birkhoff and S.~A. Kiss.
\newblock A ternary operation in distributive lattices.
\newblock {\em Bulletin of the American Mathematical Society}, 53:749--752,
  1947.

\bibitem{Birkhoff1948}
G. Birkhoff.
\newblock{\em Lattice Theory}, volume~25 of
{\em American Mathematical Society Colloquium Publications}, revised
 edition. 
\newblock American Mathematical Society, New York,   1948. 


\bibitem{NDFWA}
D.~M. Clark and B.~A. Davey.
\newblock {\em Natural dualities for the working algebraist}, volume~57 of {\em
  Cambridge Studies in Advanced Mathematics}.
\newblock Cambridge University Press, Cambridge, 1998.


\bibitem{Davey2002}
B.~A. Davey and H.~A. Priestley.
\newblock {\em Introduction to lattices and order}.
\newblock Cambridge University Press, New York, second edition, 2002.


\bibitem{Davey1983}
B.~M. Davey and H.~Werner.
\newblock Dualities and equivalences for varieties of algebras. 
\newblock In {\em Contributions to lattice theory ({S}zeged, 1980)},
  volume~33 of {\em Colloquia Mathematica Societatis János Bolyai}, pages
  101--275. North-Holland, Amsterdam, 1983.

\bibitem{Gratzer2003}
G. Gr{\"a}tzer.
\newblock {\em General lattice theory}.
\newblock Birkh\"auser Verlag, Basel, second edition, 1998.
\newblock New appendices by the author with B. A. Davey, R. Freese, B. Ganter,
  M. Greferath, P. Jipsen, H. A. Priestley, H. Rose, E. T. Schmidt, S. E.
  Schmidt, F. Wehrung and R. Wille.

\bibitem{Grau1947}
A.~A.~Grau.
\newblock {Ternary Boolean algebra}.
\newblock {\em Bulletin of the American Mathematical Society}, (May
  1944):567--572, 1947.


\bibitem{Isbell1980}
J.~R. Isbell.
\newblock Median algebra.
\newblock {\em Transactions of the American Mathematical Society},
  260(2):319--362, 1980.

\bibitem{Mitschke1978}
A. Mitschke,
\newblock Near unanimity identities and congruence distributivity in equational classes.
\newblock {\em Algebra Universalis}, 8(1):29--32, 1978.

\bibitem{Nieminen1978}
J. Nieminen.
\newblock The ideal structure of simple ternary algebras.
\newblock {\em Colloq. Math.}, 40(1):23--29, 1978/79.

\bibitem{Sholander1952}
M.~Sholander.
\newblock {Trees, lattices, order, and betweenness}.
\newblock {\em Proceedings of the American Mathematical Society},
  3(3):369--381, 1952.

\bibitem{Sholander1954}
M.~Sholander.
\newblock {Medians, lattices, and trees}.
\newblock {\em Proceedings of the American Mathematical Society},
  5(5):808--812, 1954.

\bibitem{Werner1981}
H.~Werner.
\newblock A duality for weakly associative lattices.
\newblock In {\em Finite algebra and multiple-valued logic ({S}zeged, 1979)},
  volume~28 of {\em Colloquia Mathematica Societatis János Bolyai}, pages
  781--808. North-Holland, Amsterdam, 1981.

\end{thebibliography}
\end{document}